\documentclass[12pt,reqno]{amsart}
\usepackage[a4paper]{geometry}
\usepackage{amssymb,amsmath,amsthm,enumerate,verbatim,bbm}
\usepackage[mathscr]{euscript}

\DeclareMathOperator{\Ran}{Ran}

\DeclareMathOperator{\Ker}{Ker}

\DeclareMathOperator{\BMOA}{BMOA}

\newcommand{\abs}[1]{{\lvert#1\rvert}}

\newcommand{\norm}[1]{\lVert#1\rVert}

\newcommand{\bbT}{{\mathbb T}}

\newcommand{\bbD}{{\mathbb D}}

\newcommand{\calH}{{\mathcal H}}

\newcommand{\calK}{K}

\newcommand{\bC}{\mathbf C}

\numberwithin{equation}{section}
\renewcommand{\[}{\begin{equation}}
\renewcommand{\]}{\end{equation}}

\theoremstyle{plain}
\newtheorem{theorem}{\bf Theorem}[section]
\newtheorem*{theorem*}{Theorem 1.1$'$}
\newtheorem{lemma}[theorem]{\bf Lemma}

\theoremstyle{definition}

\newtheorem*{definition*}{\bf Definition}

\theoremstyle{remark}
\newtheorem*{remark*}{\bf Remark}
\newtheorem{remark}[theorem]{\bf Remark}

\DeclareFontFamily{U}{mathx}{\hyphenchar\font45}
\DeclareFontShape{U}{mathx}{m}{n}{<5> <6> <7> <8> <9> <10>
<10.95> <12> <14.4> <17.28> <20.74> <24.88> mathx10}{}
\DeclareSymbolFont{mathx}{U}{mathx}{m}{n}
\DeclareFontSubstitution{U}{mathx}{m}{n}
\DeclareMathAccent{\wc}{0}{mathx}{"71}

\newcommand{\wt}{\widetilde}
\newcommand{\wh}{\widehat}

\newcommand{\1}{\mathbbm{1}}

\begin{document}

\title[Schmidt subspaces for Hankel operators]{The structure of Schmidt subspaces of Hankel operators: a short proof}

\author{Alexander Pushnitski}
\address{Department of Mathematics, King's College London, Strand, London, WC2R~2LS, U.K.}
\email{alexander.pushnitski@kcl.ac.uk}

\author{Patrick G\'erard}
\address{Universit\'e Paris-Sud XI, Laboratoire de Math\'ematiques d'Orsay, CNRS, UMR 8628, and Institut Universitaire de France}
\email{Patrick.Gerard@math.u-psud.fr}



\begin{abstract}
We give a short proof of the main result of \cite{A}: every Schmidt subspace of a Hankel operator is the 
image of a model space by an isometric multiplier. This class of subspaces is 
closely related to nearly $S^*$-invariant subspaces, and our proof uses Hitt's theorem 
on the structure of such subspaces. 
We also give a formula for the action of a Hankel operator on its Schmidt subspace. 
\end{abstract}

\date{12 July 2019}

\maketitle

\section{Introduction and main result}\label{sec.a}

\subsection{Hankel operators}

Let $\calH^2\subset L^2(\bbT)$ be the standard Hardy space of the unit disk,
and let $P$ be the orthogonal projection onto $\calH^2$ 
in $L^2(\bbT)$. For a \emph{symbol} 
$u\in\BMOA(\bbT)$, we define the Hankel operator $H_u$ acting on $\calH^2$ by 
\[
 H_uf=P(u\overline{f}), \quad f\in \calH^2.
\label{b0}
\]
Thus, $H_u$ is an \emph{anti-linear} operator. 
Denoting by $(\cdot,\cdot)$ the standard inner product in $\calH^2$, we have
$$
(H_uz^n,z^m)=(P(u\overline{z^n}),z^m)=(u\overline{z^n},z^m)=(u,z^{m+n})=\wh u(n+m),
$$
where $n,m\geq0$ and $\wh u(\cdot)$ are the Fourier coefficients of $u$. 
Thus, $H_u$ is the anti-linear realisation of the Hankel matrix $\{\wh u(n+m)\}_{n,m\geq0}$
in the Hardy class $\calH^2$. 
In Section~\ref{sec.a5} we 
recall the relation of $H_u$  to a linear realisation of the Hankel matrix in $\calH^2$.

Our aim is to describe the \emph{Schmidt subspaces}
$$
E_{H_u}(s):=\Ker (H_u^2-s^2I), \quad s>0,
$$
as a class of subspaces in $\calH^2$. Since $H_u$ commutes with $H_u^2$,  
we see that $E_{H_u}(s)$ is an invariant subspace for $H_u$
(this is one of the advantages of working with the anti-linear realisation $H_u$). 
We give the formula for the action of $H_u$ on this subspace. 

\subsection{Model spaces and isometric multipliers}
For an inner function $\theta$ on the unit disk we use the standard notation 
$$
\calK_\theta=\calH^2\cap (\theta \calH^2)^\perp
$$
for the corresponding model space. A convenient equivalent 
description of $\calK_\theta$ is 
\[
h\in \calK_\theta  \quad \Leftrightarrow \quad h\in \calH^2 \text{ and } \overline{z}\theta \overline{h}\in \calH^2. 
\label{a2}
\]
Observe that for $h\in\calK_\theta$, the combination $\overline{z}\theta \overline{h}$ is again in $\calK_\theta$.

As usual, we denote by $Sf(z)=zf(z)$ the shift operator in $\calH^2$ and $S^*$ 
is the adjoint of $S$ in $\calH^2$. 
Recall that the significance of model spaces stems from 
Beurling's theorem which implies that a proper subspace of $\calH^2$ is invariant
under $S^*$ if and only if it is a model space.

If $p$ is an analytic function in the unit disk, we will say that $p$
is an \emph{isometric multiplier} on $\calK_\theta$, if for every $f\in \calK_\theta$
we have $pf\in \calH^2$ and 
$\norm{pf}=\norm{f}$.
In this case we denote
$$
p\calK_\theta:=\{pf: f\in \calK_\theta\}\, .
$$
We note that for a subspace $p\calK_\theta\subset\calH^2$, 
the choice of the parameters $p$ and $\theta$ in this representation is not unique. 
One can multiply $p$ and $\theta$ by arbitrary unimodular constants and one 
can also perform \emph{Frostman shifts} on $p\calK_\theta$, see
 Section~\ref{sec.frostman} for the details. 

\subsection{Main result and discussion}

\begin{theorem}\label{thm1}\cite{A}
Let $H_u$ be a bounded Hankel operator \eqref{b0} in $\calH^2$. 
Every non-trivial Schmidt subspace $E_{H_u}(s)$, $s>0$, is of the form $p\calK_\theta$, where
$\theta$ is an inner function and $p$ is an isometric multiplier on $\calK_\theta$. 
Moreover, there exists a unimodular constant $e^{i\varphi}$ such that the action of $H_u$ on this subspace is given by 
\[
H_u(ph)=se^{i\varphi} p \overline{z} \theta \overline{h}, \quad h\in \calK_\theta\, .
\label{a00}
\]
\end{theorem}

\textbf{Remarks:}

\begin{enumerate}[1.]
\item
By \eqref{a2}, the combination $\overline{z} \theta \overline{h}$ in \eqref{a00} is in $\calH^2$; in fact, it is in $\calK_\theta$. 
\item
The constant $e^{i\varphi}$ depends on the choice of the 
parameters $p$, $\theta$ in the representation $E_{H_u}(s)=p\calK_\theta$. 
In particular, by choosing a unimodular constant in the definition of $p$ and $\theta$, 
one can achieve $e^{i\varphi}=1$ in \eqref{a00}. 
\item
For an inner function $\theta$, consider the Hankel operator $H_{S^*\theta}$. 
It is not difficult to see that 
$$
\Ran H_{S^*\theta}=E_{H_{S^*\theta}}(1)=\calK_\theta
$$ 
and the action 
of $H_{S^*\theta}$ on $\calK_\theta$ is given by the anti-linear involution 
$$
H_{S^*\theta} h= \overline{z}\theta\overline{h}, \quad h\in \calK_\theta\,.
$$
Comparing with Theorem~\ref{thm1}, we see that 
such Hankel operators can be regarded as the ``simplest" ones from the point of view of 
our analysis: they have only one non-trivial Schmidt subspace  and one can choose
 $p=\1$. Here and in what follows, $\1$ is the function  identically equal to $1$ in $\calH^2$. 
\item
Denoting by $T_p$ the Toeplitz operator with the symbol $p$ (which in this case, by the 
analyticity of $p$, coincides with the operator of multiplication by $p$), we
can rewrite formula \eqref{a00} as 
$$
H_u T_p=s e^{i\varphi}T_p H_{S^*\theta} \quad \text{ on $\calK_\theta$.}
$$
\item
In \cite{A}, formula \eqref{a00} was  discussed only in the case $\theta(0)=0$. 
This case is important because condition $\theta(0)=0$ is equivalent to $\1\in\calK_\theta$
and thus to $p\in p\calK_\theta$. In fact, in this case for every $h\in \calK_\theta$ we have
\[
(ph,\overline{p(0)}p)=p(0)(ph,p)
=p(0)(h,\1)=p(0)h(0)=(ph,\1), 
\label{a3}
\]
and therefore $\overline{p(0)}p$ is the orthogonal projection of $\1$ onto
the subspace $p\calK_\theta$. 
\item
There seems to be a close analogy between Theorem~\ref{thm1} and the structure of Toeplitz eigenspaces. 
Let $v\in L^\infty(\bbT)$ and let $T_v$ be the Toeplitz operator with the symbol $v$. 
Then (see \cite{Hayashi}) all eigenspaces of $T_v$ have the form $p\calK_\theta$. 
In fact, Toeplitz operators and operators of the form $H_u^2$ satisfy similar commutation relations, 
see Remark~\ref{rmk} below.  
\end{enumerate}
\subsection{Hitt's theorem}

A closed subspace $M\subset \calH^2$ is called \emph{nearly $S^*$-invariant}, 
if  $M\not\perp\1$ and 
\[
f\in M, \quad f\perp \1 \quad \Rightarrow \quad S^*f\in M. 
\label{a1}
\]
The proof of Theorem~\ref{thm1} given in the present paper relies on the following
fundamental result by D.~Hitt \cite{Hitt} (see also \cite{Sarason}). 

\begin{theorem}\label{thm2}\cite{Hitt}
Let $M\subset \calH^2$ be a non-trivial nearly $S^*$-invariant subspace. 
Then $M=pN$, where $N\subset \calH^2$ is an $S^*$-invariant subspace
and $p$ is an isometric multiplier on $N$. 
\end{theorem}
By Beurling's theorem, $N$ is either a model space or $N=\calH^2$; in the second case $p$ must be an inner function.

A partial converse of Hitt's theorem is obvious: if $p$ is an isometric multiplier
on an $S^*$-invariant subspace $N$, and $p(0)\not=0$ (i.e. $p\not\perp\1$), then $pN$ is nearly
$S^*$-invariant. However, if $p(0)=0$, then $pN$ is not nearly $S^*$-invariant.

Hitt's theorem seems to be closely related to Theorem~\ref{thm1}. 
However, in \cite{A} the authors were unable to use Hitt's theorem directly
(even though its key ideas were used in the proof). 
The reason for this is that the Schmidt subspaces $E_{H_u}(s)$ 
are \emph{not necessarily nearly $S^*$-invariant!} 
Indeed, the weight $p$ in the representation $E_{H_u}(s)=p\calK_\theta$ 
may vanish at zero (see e.g. Example in \cite[Section 6]{A}). 

This obstacle is overcome in the present paper through the use of conformal mapping. 
More precisely, our plan of the proof is as follows. 
At the first step we consider the case $E_{H_u}(s)\not\perp\1$. 
We prove that in this case $E_{H_u}(s)$ is nearly $S^*$-invariant and use 
Hitt's theorem to obtain the representation $E_{H_u}(s)=p\calK_\theta$. 
Some additional algebra yields the formula for the action of $H_u$. 

At the second step we consider the case $E_{H_u}(s)\perp\1$. 
We choose a point $\alpha$ in 
the unit disk such that $\alpha$ is not a common zero of all elements of $E_{H_u}(s)$. 
We then use a M\"obius map $\mu$ sending $\alpha$ to $0$ and consider the associated unitary operator $U_\mu$ 
on $\calH^2$. 
It is easy to check that  $U_\mu E_{H_u}(s)$ is a Schmidt subspace
of another bounded Hankel operator $H_w$ and that the point $0$ is not a common
zero of all elements of $E_{H_w}(s)$, i.e. $E_{H_w}(s)\not\perp\1$. 
This reduces the problem to the one considered at the first step of the proof. 

The proofs of this paper are self-contained, apart from the reliance on Hitt's theorem. 
It is informed by the intuition coming from \cite{A}, and in fact we reproduce some 
simple elements of the argument of \cite{A}.

\subsection{Linear Hankel operators}\label{sec.a5}

Here we rewrite Theorem~\ref{thm1} in terms of linear
(rather than anti-linear) Hankel operators on the Hardy space. 
We follow \cite[Appendix]{A} almost verbatim.
Let $J$ be the linear involution in $L^2(\bbT)$, 
$$
Jf(z)=f(\overline{z}), \quad z\in \bbT,
$$
and let $\bC$ be the anti-linear involution in $\calH^2$, 
$$
\bC f(z)=\overline{f(\overline{z})}, \quad z\in \bbT. 
$$
For a symbol $u\in \BMOA(\bbT)$, let us define the \emph{linear} Hankel operator $G_u$ in $\calH^2$ by 
$$
G_u f=P(u\cdot Jf), \quad f\in \calH^2. 
$$
We have
$$
G_u=H_u\bC, \quad G_u^*=\bC H_u, 
$$
and so from Theorem~\ref{thm1} we obtain
\begin{theorem}
Let $s$ be a singular value of $G_u$. 
Then there exists an inner function $\theta$ and an isometric multiplier $p$ on $\calK_\theta$ such that 
\begin{align*}
\Ker (G_u^*G_u-s^2I)&=\bC(p\calK_\theta), 
\\
\Ker (G_uG_u^*-s^2I)&=p\calK_\theta. 
\end{align*}
The action 
$$
G_u:\Ker (G_u^*G_u-s^2I)\to \Ker (G_uG_u^*-s^2I)
$$
is given by 
$$
G_u\bC(pf)=sp\theta\overline{f}, \quad
f\in \calK_\theta.
$$
\end{theorem}

\subsection{Acknowledgements}
The authors are grateful to V.~Kapustin for useful discussions.

\section{The case $E_{H_u}(s)\not\perp\1$}\label{sec.b}

\subsection{Frostman shifts}\label{sec.frostman}
Let $\theta$ be an inner function; then 
(see e.g. \cite[Theorem 10]{Crofoot}) for any $\abs{\alpha}<1$ one has
$$
\calK_\theta=g_\alpha \calK_{\theta_\alpha},
$$
where 
\[
\theta_\alpha(z)=\frac{\alpha-\theta(z)}{1-\overline{\alpha}\theta(z)}, 
\quad
g_\alpha(z)=\frac{1-\overline{\alpha}\theta(z)}{\sqrt{1-\abs{\alpha}^2}},
\label{b00}
\]
and $g_\alpha$ is an isometric multiplier on $\calK_{\theta_\alpha}$. 

It follows that if $p$ is an isometric multiplier on $\calK_\theta$, then 
\[
p\calK_\theta=pg_\alpha \calK_{\theta_\alpha}, 
\label{b1}
\]
where $pg_\alpha$ is an isometric multiplier on $\calK_{\theta_\alpha}$. 
Conversely,  if 
$$
p\calK_\theta=\wt p\calK_{\wt\theta},
$$
where $p$ is an isometric multiplier on $\calK_\theta$ and $\wt p$ is an isometric multiplier on $\calK_{\wt\theta}$, 
then, again by \cite[Theorem 10]{Crofoot},
$$
\wt p=c_1 p g_\alpha, \quad \wt \theta=c_2 \theta_\alpha, 
$$
where $\alpha\in\bbD$ and $c_1$, $c_2$ are unimodular complex numbers.

\subsection{Some algebra of model spaces}
\begin{lemma}\label{lma.b1}
Let $\theta$ be an inner function in the unit disk. 
Then
$$
S^*(\calK_\theta\cap \1^\perp)=\calK_\theta\cap (S^*\theta)^\perp.
$$
\end{lemma}
\begin{proof}
It suffices to prove the identity 
$$
\calK_\theta\cap \1^\perp=S(\calK_\theta\cap (S^*\theta)^\perp).
$$
Let $h\in \calK_\theta\cap \1^\perp$. 
Write $h=Sg$, with $g=S^*h\in \calK_\theta$. Then 
$$
(g,S^*\theta)=(Sg,\theta)=(h,\theta)=0, 
$$
and so $g\in \calK_\theta\cap (S^*\theta)^\perp$. 
Conversely, let $g\in \calK_\theta\cap (S^*\theta)^\perp$. 
Write
any $f\in \calH^2$ as $f=c\1+Sw$, $w\in\calH^2$; then 
$$
(Sg,\theta f)
=
\overline{c}(Sg,\theta)+(Sg,\theta Sw)
=
\overline{c}(g,S^*\theta)+(g,\theta w)=0, 
$$
and so $Sg\in \calK_\theta$. Clearly, $Sg\perp\1$, and so $Sg\in \calK_\theta\cap\1^\perp$. 
\end{proof}

\subsection{Some identities for $H_u$}\label{sec.b2}
Hankel operators  $H_u$ satisfy the key identity
\[
S^*H_u=H_uS;
\label{b2}
\]
in fact, this identity characterises the class of all Hankel operators. 
Recalling that 
$$
SS^*=I-(\cdot,\1)\1,
$$
from \eqref{b2} and  from $u=H_u\1$ one obtains
\[
S^*H_u^2S=H_u^2-(\cdot,u)u.
\label{b3}
\]
Multiplying \eqref{b3} by $S^*$ on the right and rearranging, we arrive at
\[
S^*H_u^2-H_u^2S^*=(\cdot,\1)S^*H_uu-(\cdot,Su)u. 
\label{b4}
\]
This relation is key to checking the definition \eqref{a1} of nearly $S^*$-invariance. 
Finally, it is straightforward to check that $H_u$ satisfies 
\[
(H_uf,g)=(H_ug,f), \quad f,g\in \calH^2.
\label{b5}
\]
\begin{remark}\label{rmk}
Observe that Toeplitz operators $T_v$ on $\calH^2$ satisfy the commutation relation
$$
S^*T_vS=T_v\, ;
$$
formula \eqref{b3} can be viewed as a rank one perturbation of this relation.
\end{remark}

\subsection{Proof of the representation $E_{H_u}(s)=p\calK_\theta$ in the case $E_{H_u}(s)\not\perp\1$}\label{sec.b3}

Here we assume that $E_{H_u}(s)\not\perp\1$ and prove the first part of Theorem~\ref{thm1}.

1) 
For $f\in E_{H_u}(s)\cap\1^\perp$ and $g\in E_{H_u}(s)$ we have
$$
(S^*H_u^2f,g)-(H_u^2S^*f,g)
=
s^2(S^*f,g)-(S^*f,H_u^2g)
=
s^2(S^*f,g)-s^2(S^*f,g)=0
$$
and therefore, by \eqref{b4}, 
$$
(f,Su)(u,g)=0.
$$
By assumption, there exists an element $h\in E_{H_u}(s)$ with 
$(h,\1)\not=0$. Take $g=H_uh$; then, using \eqref{b5},
$$
(u,g)=(H_u\1,g)=(H_ug,\1)=(H_u^2h,\1)=s^2(h,\1)\not=0,
$$
and so $(f,Su)=(S^*f,u)=0$. 
Now applying \eqref{b4} to $f$, we find
$$
(H_u^2-s^2 I)S^*f=0,
$$
i.e. $S^*f\in E_{H_u}(s)$. 
Putting this together, we see that we have checked the inclusion 
\[
S^*(E_{H_u}(s)\cap \1^\perp)\subset E_{H_u}(s)\cap u^\perp. 
\label{d3a}
\]

2) 
By Hitt's theorem,  $E_{H_u}(s)=pN$, where $p$ is an isometric multiplier on $N$
and $N$ is either a model space or $N=\calH^2$; 
we need to eliminate the second possibility. 
Suppose $N=\calH^2$. 
Since by assumption $E_{H_u}(s)\not\perp\1$, we see that $p(0)\not=0$. 
Then 
$$
p\calH^2\cap \1^\perp=zp\calH^2. 
$$
It follows that 
$$
S^*(p\calH^2\cap \1^\perp)=p\calH^2. 
$$
Comparing with \eqref{d3a}, we conclude that $u\perp E_{H_u}(s)$. 
Then for any $h\in E_{H_u}(s)$, 
$$
s^2(h,\1)=(H_u^2 h,\1)=(H_u\1,H_uh)=(u,H_uh)=0,
$$
and so $E_{H_u}(s)\perp\1$, contrary to our assumption. 

\subsection{Proof of \eqref{a00} in the case $E_{H_u}(s)\not\perp\1$}

Here we assume $E_{H_u}(s)\not\perp\1$ and prove formula \eqref{a00} for the action of $H_u$ on $E_{H_u}(s)=p\calK_\theta$.

1)
Let us first assume that $\theta(0)=0$; then 
$\1\in \calK_\theta$,  $p\in p\calK_\theta$ and by 
\eqref{a3} the element $\overline{p(0)}p$ is the orthogonal projection of $\1$ onto $p\calK_\theta$. 
Next, let $u_s$ be the orthogonal projection of $u$ onto $E_{H_u}(s)$. 
Since $H_u$ commutes with $H_u^2$, and therefore
with the operator of the orthogonal projection onto $E_{H_u}(s)$, we see that 
\[
u_s=H_u(\overline{p(0)}p)=p(0)H_up.
\label{d9}
\] 
Further, by \eqref{d3a}, we have
\[
S^*(p\calK_\theta \cap p^\perp)\subset p\calK_\theta\cap u_s^\perp.
\label{d3aa}
\]
Using Lemma~\ref{lma.b1}, 
$$
S^*(p \calK_{\theta}\cap p^\perp)
=
S^*(p(\calK_{\theta}\cap\1^\perp))
=
pS^*(\calK_{\theta}\cap\1^\perp)
=
p(\calK_\theta\cap(S^*\theta)^\perp)
=
p\calK_\theta\cap (pS^*\theta)^\perp.
$$
Comparing this with \eqref{d3aa}, we obtain 
\[
u_s=cpS^*\theta
\label{d7}
\]
with some constant $c$. Putting this together with \eqref{d9}, we get
\[
H_up=\frac{c}{p(0)}pS^*\theta.
\label{d8}
\]
In order to evaluate $c$, 
let us compute the norms on both sides of \eqref{d7}:
\begin{align*}
\norm{u_s}^2&=\norm{p(0)H_u p}^2=\abs{p(0)}^2
(H_u^2p,p)=s^2\abs{p(0)}^2(\1,\1)=s^2\abs{p(0)}^2,
\\
\norm{cpS^*\theta}&=\abs{c}\norm{pS^*\theta}=\abs{c}\norm{S^*\theta}=\abs{c}\norm{\overline{z}\theta}=\abs{c}.
\end{align*}
It follows that $\abs{c}=s\abs{p(0)}$. Substituting this into \eqref{d8}, we obtain 
$$
H_u p=se^{i\varphi}p\overline{z}\theta
$$
with some unimodular complex number $e^{i\varphi}$. 
This is exactly the required formula \eqref{a00} for $h=\1$. 
From here we easily get 
formula \eqref{a00} for a general $h\in \calK_\theta$:
$$
H_u(p h)
=P(\overline{h}u\overline{p})
=P(\overline{h}P(u\overline{p}))
=P(\overline{h}H_u p)
=se^{i\varphi}P(\overline{h}p\overline{z}\theta)
=s e^{i\varphi}\overline{h}p\overline{z}\theta. 
$$

2) 
Now let $E_{H_u}(s)=p\calK_{\theta}$, with $\theta(0)\not=0$. 
Choose $\alpha=\theta(0)$ and write
$$
p\calK_\theta=pg_\alpha \calK_{\theta_\alpha}
$$
according to \eqref{b1}.
Since $\theta_\alpha(0)=0$, by the previous step of the proof we have
$$
H_u(pg_\alpha h)=se^{i\varphi} \overline{z}pg_\alpha \theta_\alpha\overline{h}, 
\quad h\in \calK_{\theta_\alpha}.
$$
Directly from the definitions \eqref{b00} one has
$$
g_\alpha \theta_\alpha=-\theta \overline{g_\alpha}, 
$$
and so, denoting $g_\alpha h=v\in \calK_\theta$, we obtain 
$$
H_u(pv)=-se^{i\varphi}\overline{z}\theta\overline{v},\quad v\in \calK_\theta,
$$
as required.

\qed

\section{The case $E_{H_u}(s)\perp\1$}\label{sec.c}

\subsection{Conformal maps}\label{sec.c1}

For $\alpha\in\bbD$, let $\mu:\bbD\to\bbD$ be the conformal map
$$
\mu(z)=\frac{\alpha-z}{1-\overline{\alpha}z}, 
$$
and consider the corresponding unitary operator on the Hardy class,
$$
U_\mu f(z)=\frac{\sqrt{1-\abs{\alpha}^2}}{1-\overline{\alpha}z}f(\mu(z)), \quad z\in\bbD.
$$
Observe that $\mu$ is an involution, $\mu\circ\mu=\text{id}$ and $U_\mu^2=I$. 

\begin{lemma}\label{lma.c1}
Let $U_\mu$ be as defined above, and let $u\in\BMOA(\bbT)$. 
Then 
$$
U_\mu H_uU_\mu=H_{w}, \quad\text{ where }\quad w=-S^*((Su)\circ\mu). 
$$
\end{lemma}
\begin{proof}
Computing the Jacobian of the change of variable $e^{it}\mapsto \mu(e^{it})$ on the unit circle, we get for
$h_1,h_2\in \calH^2$
$$
(H_uh_1,h_2)=(u,h_1h_2)=(u\circ \mu,(h_1\circ \mu) (h_2\circ \mu) \tfrac{1-\abs{\alpha}^2}{\abs{1-\overline{\alpha}z}^2}). 
$$
Writing for $\abs{z}=1$
$$
\frac{1-\abs{\alpha}^2}{\abs{1-\overline{\alpha}z}^2}
=
-\left(\frac{\sqrt{1-\abs{\alpha}^2}}{1-\overline{\alpha}z}\right)^2
\overline{(\overline{z}\mu(z))}, 
$$
we get 
\begin{multline*}
(H_uh_1,h_2)=-\bigl(\overline{z}\mu(z)u\circ \mu, (U_\mu h_1)(U_\mu h_2)\bigr)
=
\bigl(w, (U_\mu h_1) (U_\mu h_2)\bigr)
\\
=
(H_{w}U_\mu h_1, U_\mu h_2)
=(U_\mu H_{w}U_\mu h_1,h_2). \qedhere
\end{multline*}
\end{proof}

\begin{lemma}\label{lma.c2}
Let $\theta$ be an inner function and let $p$ be an isometric multiplier on $K_\theta$. 
Then $U_\mu(p\calK_\theta)=(p\circ\mu)\calK_{\theta\circ\mu}$, 
and $p\circ\mu$ is an isometric multiplier on $\calK_{\theta\circ\mu}$. 
\end{lemma}
\begin{proof}
Clearly, $U_\mu(p \calK_\theta)=(p\circ\mu)U_\mu(\calK_\theta)$. Also,  $U_\mu(\theta \calH^2)=(\theta\circ\mu) \calH^2$ 
and so $U_\mu \calK_\theta=\calK_{\theta\circ\mu}$. 
\end{proof}

\subsection{Proof of Theorem~\ref{thm1} in the case $E_{H_u}(s)\perp\1$}

Let us choose $\alpha\in\bbD$ which is not a common zero of all elements of  $E_{H_u}(s)$. 
Consider the conformal map $\mu$ and the unitary operator $U_\mu$ corresponding to this point $\alpha$. 
By the choice of $\alpha$, the point $z=0$ is not a common zero for $U_\mu E_{H_u}(s)$, i.e. $U_\mu E_{H_u}(s)\not\perp\1$.
Moreover, by Lemma~\ref{lma.c1}, the latter subspace is a Schmidt subspace of a Hankel operator $H_w$,
$$
U_\mu E_{H_u}(s)
=\Ker (U_\mu H_u^2 U_\mu-s^2I)
=E_{H_w}(s)\, .
$$
Thus, by the already proven case of Theorem~\ref{thm1}, applied to $H_w$, 
we obtain that 
$$
E_{H_w}(s)=p\calK_\theta, 
$$
where $p$ is an isometric multiplier on $\calK_\theta$, and  that $H_w$ acts on $E_{H_w}(s)$ according
to the formula \eqref{a00}: 
\[
H_w(ph)=se^{i\varphi}p\overline{z}\theta \overline{h}, \quad h\in \calK_\theta.
\label{d6}
\]
By Lemma~\ref{lma.c2} we obtain 
$$
E_{H_u}(s)=U_\mu E_{H_w}(s)=U_\mu(p\calK_\theta)=(p\circ\mu)\calK_{\theta\circ\mu}\, ,
$$
which proves the first part of the theorem. It remains to check formula \eqref{a00} for the action of $H_u$.

Denote $U_\mu h=v\in \calK_{\theta\circ\mu}$.
Let us apply $U_\mu$ on both sides of \eqref{d6}. 
For the left hand side, we have
$$
U_\mu H_w(ph)
=
U_\mu H_w(pU_\mu v)
=
U_\mu H_wU_\mu ((p\circ\mu)v)
=
H_u((p\circ\mu)v).
$$
For the right hand side, we have
$$
U_\mu(se^{i\varphi} p\overline{z}\theta\overline{h})
=
se^{i\varphi}(p\circ\mu)(\theta\circ\mu)U_\mu(\overline{z}\overline{h}). 
$$
By the definition of $U_\mu$, 
\begin{align*}
U_\mu(\overline{z}\overline{h})
&=
\frac{\sqrt{1-\abs{\alpha}^2}}{1-\overline{\alpha}z}
\overline{\mu(z)}
\overline{h(\mu(z))}
=
\frac{\sqrt{1-\abs{\alpha}^2}}{1-\overline{\alpha}z}
\frac{\overline{\alpha}-\overline{z}}{1-\alpha\overline{z}}
\overline{h(\mu(z))}
\\
&=
-\overline{z}
\frac{\sqrt{1-\abs{\alpha}^2}}{1-\overline{\alpha}z}
\frac{1-\overline{\alpha}z}{1-\alpha\overline{z}}
\overline{h(\mu(z))}
=
-\overline{z}\overline{U_\mu h}\, .
\end{align*}
Putting this together, we obtain 
$$
H_u((p\circ\mu)v)=-se^{i\varphi}(p\circ\mu)\overline{z}(\theta\circ\mu)\overline{v}, 
$$
for all $v\in \calK_{\theta\circ\mu}$. 
This is the required formula \eqref{a00}.


\end{document}